\documentclass[11pt]{article}
\usepackage{hyperref}
\usepackage[english]{babel}
\usepackage{amssymb,amsmath, amsthm}
\usepackage{graphicx}
\usepackage{graphics}
\usepackage{psfrag}

\usepackage[all]{xy}
\usepackage[applemac]{inputenc}

\newtheorem{thm}{Theorem}[section]
\newtheorem{lemma}{Lemma}[section]

\newtheorem{prop}{Proposition}[section]

\theoremstyle{definition}
\newtheorem{rem}{Remark}[section]
\newtheorem{rems}{Remarks}[section]

 \author{Mário M. Gra\c{c}a and M. Esmeralda Sousa-Dias   \thanks{Departamento de Matem\'{a}tica,
Instituto Superior T\'ecnico, Universidade Técnica de Lisboa, Av. Rovisco Pais, 1049--001 Lisboa, Portugal. Corresponding author: mgraca@math.ist.utl.pt.}}

\begin{document}
 
\title{A unified framework for  the computation of polynomial quadrature  weights and errors\footnote{Dedicated to the memory of Bernard Germain-Bonne (1940--2012).}}
\maketitle

%
%

 \begin{abstract}

\noindent
 For the class of polynomial quadrature rules we show that  conveniently chosen bases allow to compute both the weights and the  theoretical error expression of a $n$\--point rule via the undetermined coefficients method. As an illustration, the framework  is  applied  to some classical  rules such  as Newton\--Cotes, Adams\--Bashforth, Adams\--Moulton and  Gaussian rules. 
\noindent

   \end{abstract}



\section{Introduction}\label{introd}

Most  numerical integration schemes widely used in the applications are based on the class of polynomial quadrature.
Here we present   an unified framework for the simultaneous computation of weights and theoretical error of a $n$-point polynomial quadrature rule $Q_n(f)$.  
We aim to bring together the computation of the pair $\text{(weights, error)}=(W_n, E_n)$ under a simple and unified framework. Our approach shows that it is possible to choose a particular polynomial basis  relative to which the weights $W_n$  are the solution of an upper triangular linear system,  enabling so their   recursive computation. At the same time, by extending conveniently  the referred basis,  we obtain a (computable)  criterium for   the degree of precision  of the rule and consequently  the respective error expression $E_n$.  

The  upper triangular system, whose  solution is the vector $W_n$ of  weights, is obtained by  the technique known in numerical analysis as undetermined coefficient method (not to be confused with the same named method  for ODE's).  The undetermined coefficients method (UCM) is often employed  for   weights's computation, however this method has been  set aside  for  the determination of  $E_n$  (see for instance the remarks in Gautschi \cite{gautschi}, p.~176, and Evans \cite{evans}, p.~69).  To the best of our knowledge,  our  scheme for the  simultaneous computation of the  weights and error via the UCM   is new.

In  numerical quadrature one aims     to approximate an integral  $I(f)=\int_a^b f(x) dx$ by a $n$-point rule $Q_n(f)=\sum_{k=1}^n a_i f(x_i)$,  and   to obtain an expression or a bound for the error  $E_n(f)=I(f)-Q_n(f)$. The coefficients $a_i$ in $Q_n(f)$ are called the (quadrature) weights and the points $x_i$ in a given set $\{x_1, \ldots, x_n\}$ are usually called nodes (a.k.a.  quadrature points, or abscissae).
 The rule $Q_n(f)$ is said to be  polynomial  interpolatory  when it is exact  for polynomials of degree less or equal to $n-1$. That is, when $Q_n(p)=I(p)$ for all polynomials $p$ of degree less or equal to $n-1$.
 Hereafter when we  refer to a rule we mean a polynomial   quadrature rule.

Quadrature has a long history and is  a classical subject, although it is still nowadays an active area of research  (general references are, for instance,  Atkinson \cite{atkinson},  Davis and Rabinowitz \cite{davis}, Gautschi \cite{gautschi} and Krylov \cite{krylov}).  There are  {\it a priori}  two main goals when dealing with numerical quadrature: the determination of the weights and  respective  theoretical error expression. Traditionally, in order to obtain the  weights,  one first   determines the  interpolating polynomial for a  given  panel $\{(x_1,f(x_1)), (x_2,f(x_2)) \ldots, (x_n,f(x_n))\}$  and afterwards this polynomial is  integrated. The interpolating polynomial is unique (and so are the weights of the rule) but  it may be written in several different forms   named after  Lagrange, Newton, Hermite, Chebyshev, etc. Of course, this means  that we can always write the interpolating polynomial in distinct  polynomial bases. 

 In the classical  literature  the theoretical  quadrature error is usually deduced from a particular  expression representing  the interpolation error \--- see Steffenson \cite{steffensen} for  a survey on interpolation theory,  and  Berezin and Zhidkov \cite{berezin}, Ch. 2,  for its applications. Consequently, the  quadrature error deduction is  based upon   a panoply of analytical tools such as mean value theorems, integration by parts, Peano kernels, Euler-MacLaurin formula and so on.  The usual approach for the computation of the pair $(W_n, E_n)$,  where $W_n=(a_1, \ldots, a_n)$,  is   done on a rule by rule basis leading usually to lengthy and cumbersome calculations even for a small number $n$ of nodes. 

At the heart of our approach are  conveniently chosen  bases of polynomials relative to which we apply the undetermined coefficients method  in order to compute  the pair $(W_n, E_n)$. Only elementar results from  linear algebra and numerical analysis are needed.

The paper is organized as follows. There are two sections, in the first one we prove the main result (Theorem~\ref{prop1}), which  points out   an algorithm for computing the pair $(W_n,E_n)$.  In the following section we illustrate its   applicability to  well known quadrature rules named after their creators: Newton-Cotes, Adams-Basforth, Adams-Moulton, and Gaussian rules. This shows the wide applicability of the algorithm previously referred.  In particular,  when applying  Theorem~\ref{prop1} to  Gaussian quadratures  we recover (in Proposition \ref{error Gauss}) a well known   closed form for their error expression.

We  barely address here numerical computational issues since they are out of the scope of this work, and will be treated elsewhere. However, several symbolic/numerical tests were carried out in order to compute   pairs $(W_n,E_n)$ for the  illustrative quadrature rules mentioned above. For, we  developed   a {\sl Mathematica} code translating the algorithm in  Theorem~\ref{prop1}   and we have  tested it on all the  rules described in Section~\ref{sec3}. The  respective computed weights and error expressions    were compared with those    in   tables spread in the literature. In particular, for a number of nodes   $2\leq n\leq 256$, the algorithm has been used to   compute  $100$\--digits precision  pairs $(W_n,E_n)$ for the ubiquitous and indispensable  Gauss-Legendre rule.

 \section{Weights and error for a polynomial    rule}\label{sec1}
 
 In order to approximate the integral $I(f)=\int_a^b f(x) dx$, where $f$ is assumed to be an integrable  function in $(a,b)$, and of class $C^k ([a,b])$ (where the integer $k$ will be taken   appropriately  to each case), one  constructs the  rule 
\begin{equation}\label{rule}
Q_n(f)= a_1f(x_1)+a_2 f(x_2)+\cdots+a_n f(x_n),
\end{equation}
from a given set  $\{x_1,x_2, \ldots, x_n\}$, with $n\geq1$.

The  rule $Q_n(f)$ is said to have degree of precision $d$,  if $I(p)=Q_n(p)$ for all polynomials $p$ of degree $\leq d$, and $I(p)\neq Q_n(p)$ for some polynomial $p$ of degree $d+1$.
 If $d=\operatorname{deg} (Q_n(f))$ denotes  the degree of precision of the rule $Q_n(f)$, and $f$ is of class $C^{d+1}$ in $[a,b]$, the quadrature error is given by 
\begin{equation}\label{error}
E_n(f)= I(f)-Q_n(f)= \frac{I(p_{d+1})-Q_n(p_{d+1})}{(d+1)!} f^{(d+1)}(\xi),
\end{equation}
where $\xi$ and  $p_{d+1}$ are respectively  a certain point in $(a,b)$ and  a  polynomial  of degree $d+1$ (cf. Gautschi \cite{gautschi}, p. 178).
We remark   that the degree of precision of a rule cannot be greater than $(2n-1)$ since, if $p$ is a polynomial of degree $n$ then $I(p^2)$ is positive.

In numerical analysis the undetermined coefficients method  consists in solving the  linear system arising from the equalities $Q_n(g_i)=I(g_i)$, for $0\leq i\leq (n-1)$, where $g_i$ are the elements of a prescribed set of functions. In the  context of polynomial quadrature  this set  will be a set of polynomials.
The next theorem   shows   that when the UCM is applied to a conveniently chosen basis of polynomials,   the weights  are obtained  recursively  and, at the same time,     the degree of  precision $d$ of the corresponding quadrature rule is easily  found.

 \medskip
We denote by $\mathbb{P}_{s}$  the linear space of real  polynomials of degree less or equal to $s$. Suppose  that  $x_1,x_2, \ldots, x_n$ are  $n\geq 1$ distinct  real   points and consider the   basis $B_0=\{\varphi_0(x), \varphi_1(x), \ldots, \varphi_{n-1}(x)\}$ of $\mathbb{P}_{n-1}$,  given by
 \begin{equation}\label{basis fi}
 \left\{\begin{array}{l}
 \varphi_0(x)=1\\
 \varphi_j(x)=\varphi_{j-1}(x) (x-x_j), \quad 1\leq j\leq n-1.
 \end{array}\right.
 \end{equation}
 The basis $B_0$ is known in some literature as Newton's basis (see Steffensen~\cite{steffensen}, p.~23). 
 We complete the  basis $B_0$  with other polynomials $q_n(x), q_{n+1}(x), \ldots, q_{n+k}(x)$ to form a basis  $\mathbb{P}_{n+k}$. Notably,
\begin{equation}\label{def q}
\left\{\begin{array}{l}
 q_n(x)=\varphi_{n-1}(x) (x-x_n)\\
 q_{j}(x)=q_{j-1}(x) (x- x_r), \quad j\geq n+1\quad \mbox{with}\quad  r\equiv j \pmod{n},
 \end{array}\right.
 \end{equation}
 where $k$ is a nonnegative integer which  will be later  specified in accordance with the particular quadrature rule under consideration. Note that the  points $x_i$ ($1\leq i\leq n-1$) are roots  for   all the    polynomials $q_j$ defined in \eqref{def q}.

 \begin{thm}\label{prop1} Let   $x_1, \ldots, x_n$ be  $n\geq 1$ distinct  real quadrature points, $B_0=\{\varphi_0(x), \ldots, \varphi_{n-1}(x)\}$ the basis of $\mathbb{P}_{n-1}$ defined in \eqref{basis fi}, and $B=B_0\cup\{q_n(x), \ldots, q_{n+k}(x) \}$ the basis of $\mathbb{P}_{n+k}$,  defined  in \eqref{def q}.
 
 \begin{itemize}
 \item[(a)] The undetermined coefficient method applied to the basis $B_0$ determines  (uniquely)  the weights  $a_i$ of  the  rule \eqref{rule} for approximating the integral $I(f)=\int_a^bf(x) dx$.
These weights  are 
\begin{equation}\label{coef prop1}
\left\{\begin{aligned}
a_n&=\frac{I(\varphi_{n-1})}{\varphi_{n-1}(x_{n})},\\ 
a_i&=\displaystyle{\frac{I(\varphi_{i-1})-\sum_{k=i+1}^n\varphi_{i-1}(x_{k}) a_k}{\varphi_{i-1}(x_{i})}}, \quad i=n-1, n-2, \ldots, 1.
\end{aligned}\right.
\end{equation}
\item[(b)]The undetermined coefficient method applied to the basis $B$, determines   the degree of precision $d=\operatorname{deg} (Q_n(f))$ as being the  integer $d\geq n-1$ for which
\begin{equation}\label{degree cond}
I(q_{d+1})\neq0 \quad \text{ and}\quad\text{ $I(q_j)=0$, for all $\,\,n\leq j\leq d$.}
 \end{equation}
\item[(c)]  If    $f$ is of class $C^{d+1}([a,b])$,   the error expression is
 \begin{equation}\label{error prop}
 E_{Q_n}(f)= c_n  f^{(d+1)}(\xi), \quad\text{with $\,\,c_n=\frac{I(q_{d+1})}{(d+1)!}$},
 \end{equation}
 for some $\xi\in(a,b)$.
 \end{itemize}
 \end{thm}
\begin{proof} (a) As the rule $Q_n(f)$ is polynomial, applying the UCM to  the basis $B_0$ of the polynomials $\varphi_i$ of degree $i\leq (n-1)$, we have
\begin{equation}\label{eq 7}
Q_n(\varphi_i)= I(\varphi_i), \quad\text{ for all  $\,\,0\leq i\leq (n-1)$}.
\end{equation}
The conditions \eqref{eq 7} are equivalent to the linear system $Ax=b$ in the unknowns $a_1, \ldots, a_n$, where $b=(I(\varphi_0), I(\varphi_1), \ldots, I(\varphi_{n-1}))^T$. The matrix $A$ is a $n\times n$  upper triangular matrix whose  diagonal entries are $1, \varphi_1(x_2),  \ldots$, $\varphi_{n-1}(x_n)$.  Furthermore,  $A$ is obviously nonsingular since $x_{i+1}$ is not a root of $\varphi_i$, for $1\leq i\leq n-1$, and so $\det(A)\neq0$. The (unique)  solution of $Ax=b$ is obtained by backward substitution and is given by \eqref{coef prop1}.
 
(b)  Applying the UCM to the basis $B$, we get an overdetermined  linear system $\tilde{A} x=\tilde{b}$, where $x$ is  the vector whose components are  the weights, and the $(n+k)\times n$ matrix   $\tilde {A}$ and the vector $\tilde{b}$ are,  respectively,   of the form
 \begin{equation}\label{tildeA}
 \tilde{A}=\begin{bmatrix}
 A\\
 \hline
 O
 \end{bmatrix}_{(n+k)\times n} \qquad \tilde{b} = (I(\varphi_0),\ldots, I(\varphi_{n-1}), I(q_n),\ldots,I(q_{n+k}))^T,
\end{equation}
where $O$ denotes the zero matrix. The conditions \eqref{degree cond} in the statement are equivalent to say that $d= \operatorname{deg} (Q_n(f))$ is the greatest  integer $n+k$ for which the system $\tilde{A}x=\tilde{b}$ is compatible. 

Let us now show that  the degree $d$ of the rule $Q_n(f)$ is given by the conditions \eqref{degree cond}.
 Consider  $p$ to  be a polynomial of degree $n+k$ for some nonnegative integer $k$ and write $p$ as  (unique) linear combination of the  elements of the basis $B$ of $\mathbb{P}_{n+k}$. That is,
 $$p =\sum_{j=0}^{n-1} \alpha_j\varphi_j+\sum_{j=n}^{n+k}\alpha_j q_j.
 $$
 Then,
 \begin{equation}\label{th q}
 Q_n(p)= \sum_{j=0}^{n-1} \alpha_j Q_n(\varphi_j)+\sum_{j=n}^{n+k}\alpha_j Q_n(q_j)= \sum_{j=0}^{n-1} \alpha_j Q_n(\varphi_j),
 \end{equation}
 where  the last equality follows from the  fact that all  points  $x_i$ are roots of the polynomials $q_j$. Furthermore, by linearity of the operator $I$, we have
 \begin{equation}\label{th int}
I(p)= \sum_{j=0}^{n-1} \alpha_j I(\varphi_j)+\sum_{j=n}^{n+k}\alpha_j I(q_j).
\end{equation}
As by  \eqref{eq 7}, $Q_n(\varphi_j)=I(\varphi_j)$,  we conclude  from \eqref{th q} and \eqref{th int} that 
 \begin{equation}\label{eq 8}
 I(p)- Q_n(p)=\sum_{j=n}^{n+k}\alpha_j I(q_j).
 \end{equation}
Therefore the degree of precision $d$ cannot be less than $(n-1)$. As there exists a polynomial $p_s$ of  degree $s\leq 2n$ such that $I(p_s)\neq 0$, this implies that $I(q_s)\neq0$ for an integer $s$ such that  $n\leq s\leq 2n$.  It  also follows easily  from \eqref{eq 8} that $Q_n(p)=I(p)$, for all polynomials $p$ of degree $\leq d$,  if and only if $I(q_j)=0$ for all $n\leq j\leq d$. So,  the conditions  \eqref{degree cond} hold.

(c)  As    all the quadrature points are roots of the polynomial $q_{d+1}$,  we have  $Q_n(q_{d+1})=0$.  Then, it follows from   \eqref{error}  that the error's  expression is given by  \eqref{error prop}.

\end{proof}
\begin{rems}\label{rem21}

 \begin{itemize}
\item[1)] It becomes   clear from the proof of Theorem~\ref{prop1}  the reason why the UCM has been set aside for the computation of the degree of a rule. Indeed, if one had applied the UCM to another basis $B$, for instance the canonical basis of polynomials,  the matrix $\tilde{A}$ in \eqref{tildeA} would not have the zero block (or equivalently the second term of the sum  in   \eqref{th q} does not vanishes), and so it will be impossible to obtain the criterium \eqref{degree cond} for the degree.
\item[2)] Implicit in the proof of the above theorem is the following algorithm. Given the nodes $x_1,\ldots, x_n$ and the bases $B_0$ and $B$;  (i) Compute the  weights $W_n$ using  the recursion relation \eqref{coef prop1}; (ii) Find the first indice $j\geq n$ such that  $I(q_j)\neq 0$; (iii) $d=j-1$ is the degree of the rule; (iv) Compute the coefficient $c_n$ in \eqref{error prop}; (v) Output $W_n$ and $E_n$.
\item[3)]   As the degree of precision of a rule cannot be greater that $2n$, it  is  only  necessary to extend the basis $B_0$ to a basis $B$ of $\mathbb{P}_{2n}$.
  It also follows from the proof of Theorem~\ref{prop1} that the polynomials $q_j$  only need to satisfy the requirement of having   all quadrature points as roots. 
\item[4)] From   Theorem~\ref{prop1}, we can conclude  that the quadrature weights $a_i$ of a polynomial rule do not depend wether  the quadrature points belong or not to the integration interval $[a,b]$. However, the value of the integrals $I(q_j)$   do depend on the points's  location relative to the  interval of integration. In the next section we show how the position of the quadrature points may  influence  the  degree of the precision  of  a rule, for instance in the cases of the  Newton-Cotes and   Adams-Bashforth-Moulton rules.
\item[5)] The upper triangular system leading to the computation of the weight's vector $W_n$, has determinant equal to $\det(A) =\prod_{j=1}^{n-1} \varphi_j(x_{j+1})$. For $n$ large and certain sets of nodes we can have $\det(A)\approx 0$. In this case we are facing an ill-conditioned problem which  raises challenging computational issues.
\item[6)] It would  be interesting to apply   Theorem~\ref{prop1}  to     polynomial rules other  than those treated in Section~\ref{sec3}, such as the Clenshaw-Curtis, Patterson, Gauss-Konrad, etc. (see Evans~\cite{evans} for a survey  of other available rules, and references therein). 
\end{itemize}
\end{rems}

\section{Applications to some classical integration rules}\label{sec3}
In this section we apply  Theorem~\ref{prop1} to some classical  rules in order to obtain their weights and the respective error's expression. To illustrate the wide applicability of  our  result we begin with   rules  with equally spaced nodes (Newton-Cotes    and  Adams-Bashforth-Moulton)  followed by   rules  with unequally spaced  points (Gaussian).

\subsection{Weights and error for  Newton-Cotes   and Adams-Bash\-forth-Moulton rules}

The abscissae  for both Newton-Cotes (open or closed) rules and Adams-Bashforth-Moulton  are equally spaced. In the rules of the first type  all the quadrature points belong to the integration interval $[a,b]$ whilst  for the second type there are points outside $[a,b]$.  These particularities  of the nodes  have a decisive influence on the rule's  degree. The next Lemma shows how the  abscissae's distribution  relative  to the integration interval will affect  the values of the integrals $I(q_j)$ in Theorem~\ref{prop1}. Although the results in  Lemma~\ref{lema}  are  scattered  in the literature under somehow different formulations,  we include a proof for the sake of completeness.

\begin{lemma}\label{lema} Let $y_1, y_2, \ldots, y_n$ be real points  such that $y_i=y_{i-1}+h$, for some positive constant $h$, and $q_n(x)= (x-y_1) (x-y_2)\cdots (x-y_n)$.
\begin{itemize}
\item[i)] If $y_1=a$ and $y_n=b$, then 
$$\int_a^b q_n(x) dx=\begin{cases}
0&\quad\text{if $n$ is odd}\\
\neq 0&\quad\text{if $n$ is even}.
\end{cases}$$
\item[ii)] If $q_{n+1}(x)=q_n(x) (x-\tilde{y})$, with $\tilde{y}=y_1$ or $\tilde{y}\notin [y_1, y_{n}]$, then  $\int_{y_1}^{y_{n}} q_{n+1}(x) dx\neq0$.
\item[iii)] If $y_{n-1}=a$ and $y_{n}=b$, then $\int_a^b q_n(x) dx\neq0$.
\end{itemize}
\end{lemma}
\begin{proof} The  change of variables $\gamma: x\mapsto t$, defined by $x=y_1+h\left(t+\frac{n-1}{2}\right)$, mapps   $y_1\mapsto -\frac{n-1}{2}, \, y_2\mapsto -\frac{n-1}{2}+1, \ldots, y_{n-1}\mapsto \frac{n-1}{2}-1, y_n\mapsto \frac{n-1}{2}$, and $[a,b]\mapsto [-(n-1)/2, (n-1)/2]$.

If $n$ is odd, the integral $\int_a^bq_n(x) dx$ is equal to the integral of an odd function and so it is zero.That is,
$$\int_a^b q_n(x) dx=h^{n+1}\int_{-\frac{n-1}{2}}^{\frac{n-1}{2}} t (t^2- 1) (t^2-2^2)\cdots \left(t^2-\frac{(n-1)^2}{4}\right) dt =0.$$
 If $n$ is even, then $\int_a^bq_n(x) dx$ is equal to the integral of an even function:
$$\int_a^b q_n(x) dx=h^{n+1}\int_{-\frac{n-1}{2}}^{\frac{n-1}{2}}  (t^2- 1/4) (t^2-9/4)\cdots \left(t^2-\frac{(n-1)^2}{4}\right) dt \neq 0.$$
For ii), we can write $\tilde{y}$ as $\tilde{y}= y_1+ h \left(\frac{n-1}{2}\right)+ r$, where $r$ is some nonzero  constant. Thus, $(x-\tilde{y})= (x-(y_1+ h \left(\frac{n-1}{2}\right))+ r$, and so
\begin{align*}
\int_{y_1}^{y_{n}} q_{n+1}(x) dx&=\int_{y_1}^{y_{n}} q_n(x) \left[x-\left(y_1+ h \left(\frac{n-1}{2}\right)\right)\right] dx- r \int_{y_1}^{y_{n}} q_n(x) dx\\
&= K_1- rK_2.
\end{align*}
Applying (i) and the same change of variables as before, it is easy to conclude: (a) When $n$ is odd, then  $K_2=0$ and $K_1$ is   the integral of an even function in $[-(n-1)/2, (n-1)/2]$, and so $\int_{y_1}^{y_{n}}  q_{n+1}(x) dx\neq 0$; (b) when $n$ is even, then  $K_2\neq0$ and $K_1$ is   the integral of an odd function in $[-(n-1)/2, (n-1)/2]$,  thus $\int_{y_1}^{y_{n}}  q_{n+1}(x) dx\neq 0$.

 (iii):    When $x$ belongs to the open interval $ (y_{n-1}, y_{n})$ the product $(x-y_1) (x-y_2)\cdots(x-y_{n-1})$ is  positive and $(x-y_n)$ does not vanishes. Therefore,  the integrand function $q_n(x)$ does not change sign in $[y_{n-1}, y_{n}]$, hence $\int_{y_{n-1}}^{y_n}q_n(x)\neq 0$.
\end{proof}

\subsection*{Closed Newton-Cotes rules}

In order to approximate $I(f)=\int_a^b f(x)dx$,  the interval $[a,b]$ is divided into $(n-1)$ parts of equal length $h=\frac{b-a}{n-1}$,  and the nodes are:   $x_i=a +i h$, for $i=0,1,\ldots, (n-1)h$. By  a change of variables, one can consider  the $n$ nodes to be $t_1=0, t_2=1, \ldots, t_n =n-1$.  The integral $I(f)$ is
$$
I(f)=\int_a^b f(x)dx= h\int_0^{n-1} g(t) dt,$$
with $g(t)= f(a+th)$. The   rule $Q_n(f)$ is related to the  rule
$$Q_n(g) = b_1 g(0)+b_2 g(1)+\cdots+ b_ng(n-1)
$$
by $Q_n(f) =h\, Q_n(g)$. 
The elements of the basis  $B_0$ in \eqref{basis fi}  are $\varphi_0(t)=1$,  $\varphi_j(t)=\varphi_{j-1}(t)(t-(j-1))$, for $ j=1,2,\ldots, n-1$.  The polynomials $q_n$ and $q_{n+1}$ in  \eqref{def q}  are
$$
q_n(t)=t (t-1)\cdots  (t-(n-1)),
\quad q_{n+1}(t)= t^2 (t-1)\cdots  (t-(n-1)).
$$
By \eqref{coef prop1} of  Theorem~\ref{prop1}, the weights $b_i$  for $Q_n(g)$  are given recursively by
$$
\left\{\begin{aligned}
b_n&= \frac{I(\varphi_{n-1})}{\varphi_{n-1}(n-1)}=\frac{\int_0^{n-1}\varphi_{n-1}(t) dt}{(n-1)!},\\ 
&\\
b_i&= \frac{I(\varphi_{i-1})-\sum_{k=i+1}^n\varphi_{i-1}(k-1) b_k}{(i-1)!},
\end{aligned}\right.
$$
where we have  applied the fact that $\varphi_{k} (k)= k!$. 
The weights  for the rule $Q_n(f)$  are $a_i= h\, b_i$.
 The degree of $Q_n(g)$ is at least $(n-1)$,  and from Lemma~\ref{lema}-(i)  the integral $I(q_n(t))=\int_0^{n-1}t(t-1)\cdots (t-(n-1)) dt$ vanishes when is odd and is nonzero when $n$ is even. Also, by (ii) of the same Lemma, $I(q_{n+1}(t))\neq 0$. Then, by Theorem~\ref{prop1}, the degree $d$ and the error expression $E_{n}(f)$ for  a $n$-point Newton-Cotes rule are respectively,
$$
\begin{cases}
d=n-1, &\text{ if $n$ is even}\\
d=n,&\text{if $n$ is odd}
\end{cases}
$$
and
\begin{align*}
E_n(f)&=I(f)-Q_n(f)= h( I(g)-Q_n(g))\\
&=h \frac{I(q_{d+1})}{(d+1)!} g^{(d+1)}(\tilde{\xi})=h^{d+2} \frac{I(q_{d+1})}{(d+1)!} f^{(d+1)}(\xi),
\end{align*}
where the last equality follows from  the fact that $g^{(k)}(t) = h^k f^{k}(a+ th)$,  and we are assuming  that $f\in C^{d+1}([a,b])$. That is, 
\begin{equation}\label{errorNC}
E_n(f)=h^{d+2} c_n f^{(d+1)}(\xi),\quad \text{with $c_n= \frac{I(q_{d+1})}{(d+1)!}$}.
\end{equation}
Using an analogous  procedure, it is  now a simple exercise  to obtain the pair $(W_n,E_n)$ for a $n$-point  open Newton-Cotes rule. 
\subsection*{The Adams-Bashforth-Moulton rules}

The Adams-Bashforth (AB)  rules are  the core of  explicit methods for ODEs and the Adams-Moulton (AM)  rules play a central role  as  implicit methods for ODEs with the same name (Henrici~\cite{henrici}, Gautschi~\cite{gautschi}, Che\-ney~\cite{cheney}).

The (equally spaced) nodes  for the AB rule are less or equal to $a$ (where $[a,b]$ is the integration interval), while in the AM case the points $a$ and $b$ are quadrature points and there are no other abscissae in the interior of $[a,b]$. The Lemma~\ref{lema} and  Theorem~\ref{prop1} imply  that both the AB and AM rules have degree of precision $d=(n-1)$. 

We now describe in detail the polynomial bases  to be used and we determine the   theoretical error expression for both rules.

\subsubsection*{The Adams-Bashforth rule}

Given  a constant step $h$ and a fixed real number $\tau$,   consider  $n\geq2$ nodes
$x_1=\tau, \,\, x_2= \tau-h,\, \ldots, \, x_n =(\tau-(n-1)h)
$.
The Adams-Bashforth rule aims to approximate the integral $I(f)=\int_\tau^{\tau+h} f(x) dx$ by
$$Q_n(f)= a_1 f(x_1)+\cdots +a_n f(x_n).
$$
Making the change of variables defined by $x=\tau+ h t$,  we have
$$I(f)=\int_\tau^{\tau+h} f(x) dx=h \int_0^1 f(\tau+h t) dt=h \int_0^1 g(t) dt.$$ The nodes $x_i$ are mapped into  
$t_1=0,\,\, t_2=-1, \,\ldots, t_n= - (n-1)$.
The quadrature rule for $f$ is related to the  rule for $g$ by $Q_n(f) = h \,Q_n(g)$, with
\begin{equation}\label{rule AdamB}
Q_n(g)= b_1g(0)+b_2 g(-1)+ \cdots+ b_n g(1-n)=\sum_{k=1}^n b_k \, g(1-k).
\end{equation}
As the integral $I(q_n)=\int_0^1q_n(t) dt$, where $q_n(t) = \varphi_{n-1}(t) (t+(n-1))$, is obviously positive, then
  the degree of $Q_n(g)$ (and so the degree of $Q_n(f)$) is $d=n-1$. Therefore, the error of $Q_n(g)$ is $E_n(g)= h \frac{I(q_n(t))}{n!} g^{(n)}(\tilde{\xi})$ and so
\begin{equation}\label{error AdamB}
E_{n}(f) = h^{n+1}c_n f^{(n)}(\xi), \quad\text{with $c_n =\frac{I(q_n(t))}{n!}$,  \,\, $\xi\in(a,b)$},
\end{equation}
assuming  $f$   to be of class $C^n([a,b])$.

\subsubsection*{The Adams-Moulton rules}

Likewise,   we intend to approximate 
$I(f)=\int_\tau^{\tau+h} f(x) dx$
by the rule $Q_n(f)=\sum_{i=1}^n a_i f(x_i)$, but now the nodes are $x_1=\tau+h, \, x_2=\tau, \,x_3=\tau-h, x_4=\tau-2h, \cdots, x_n=\tau-(n-2)h$. Performing the same change of variables as in the previous rule,  we transform the interval $[2-n, 1]$ into the interval $[0,1]$, and the nodes $x_i$ into the nodes $t_1=1, t_2=0, \ldots, t_n=2-n$. As before $Q_n(f)= h\, Q_n(g)$.  Here the    polynomial $q_n(t)$ is $q_n(t)=\varphi_n(t) (t- (2-n))$ and so, by Lemma~\ref{lema}-(iii),  we obtain $I(q_n(t))=\int_0^1 q_n(t) dt\neq0$. Therefore  $\operatorname{deg} (Q_n(f))=(n-1)$ and  from Theorem~\ref{prop1} it follows an expression for $E_{n}(f)$ similar  to \eqref{error AdamB}.


\begin{rem}
In  the Adams-Bashforth-Moulton rules  the nodes $t_i$ are integers and so  the application of the  UCM leads  to a linear system $Ax=b$, where $A$ is a matrix  of integer entries. The integrals $I(\varphi(t))$ are rational numbers. Therefore both the  weights of $W_n$ and the respective error coefficient $c_n$ in \eqref{error AdamB} can be computed exactly if one uses  a computer algebra system   able to represent exactly rational numbers. Of course the same is true for the Newton-Cotes rules discussed previously.
\end{rem}

\subsection{Gaussian quadrature rules}

Gaussian rules, first introduced by Gauss in \cite{gauss}, rely on the properties of  orthogonal sets of polynomials (see Gautschi~\cite{gautschi2} for a recent account on the subject).
 A Gaussian  rule $G_n(f)=\sum_{k=1}^n a_kf(x_k)$ is used to approximate the integral $I(f)=\int_a^b w(x) f(x) dx$, where the so-called weight function $w$ is assumed to be a positive function defined in $[a,b]$, continuous, and integrable in $(a,b)$.

In the context of Gaussian quadratures the following    $L^2$ inner product plays a central role
\begin{equation}\label{inner}
\langle g,h\rangle=\int_a^b w(x) g(x) h(x) dx.
\end{equation}
The existence of  an orthogonal polynomial  basis $\{O_0, O_1, \ldots, O_n\}$ for $\mathbb{P}_n$  is guaranteed by the Gram-Schmit process,  and it is common  to call  the elements of such bases by  orthogonal polynomials.
The most widely used sets of orthogonal polynomials are  known by the names of  Hermite,  Laguerre, Jacobi, Chebyshev  and  Legendre polynomials. Lists of orthogonal polynomials, respective weight functions and integration interval  can be found,  for instance,  in Gaustchi~\cite{gautschi2}, Evans~\cite{evans} and Krylov~\cite{krylov}. 

 Gaussian quadratures are based on  the choice of the nodes  $x_1, x_2, \ldots, x_n$ as  roots of an orthogonal   polynomial $O_n$ of degree $n$. Here we do not address   the computation of these roots, we  assume  they are available.

The roots of an orthogonal   polynomial  are  simple, real, and lie in the interval $(a,b)$ (see Atkinson~\cite{atkinson} or  Gautschi~ \cite{gautschi2}). Also, as it is  easy to prove,   any orthogonal polynomial $O_j$ of degree $j$ is orthogonal to any polynomial $p$ of degree less than $j$.

The Theorem~\ref{prop1} applied to a $n$-point Gaussian  rule  enables us to recover well known results  for   the degree  and   error expression of this type of rules as shown in the next Proposition.

\begin{prop}\label{error Gauss}
Let $x_1, \ldots, x_n$  be  the roots of  an orthogonal  polynomial $O_n$ of degree $n$. Consider  $I(f)=\int_{a}^b w(x) f(x) dx$, where $w$ is a weight function and $f\in C^{2n}([a,b])$. 

A  Gaussian  quadrature rule $G_n(f)= \sum_{i=1}^n a_i f(x_i)$ for approximating $I(f)$ has degree of precision $d=2n-1$. The respective error $E_n(f)$ is
\begin{equation}\label{error gaussL}
E_{n}(f) =\frac{1}{\alpha^2} \frac{\|O_n\|^2}{(2 n)!} f^{(2n)}(\xi),
\end{equation}
where $\alpha$ is the coefficient of $x^n$ in $O_n$,  $\xi$ some point in $(a,b)$, and the norm $\|O_n\|$ is the $L^2$  norm induced by the inner product \eqref{inner}.
\end{prop}
\begin{proof}  Let  $B_0$ and $B$ be the bases,  respectively for  $\mathbb{P}_{n-1}$ and $\mathbb{P}_{2n}$,    defined by \eqref{basis fi} and \eqref{def q}. 
 By  Theorem~\ref{prop1} it is sufficient to prove that $I(q_{2n})\neq 0$ and $I(q_j)=0$ for $n\leq j\leq 2n-1$. 
 It is obvious  that $I(q_{2n})\neq 0$ since  the polynomial $q_n$ is the positive function  $q_{2n}=(x-x_1)^2(x-x_2)^2\cdots (x-x_n)^2$, and so $I(q_{2n})=\int_{a}^bw(x)q_{2n} (x) dx>0$. Hence, the degree of a $n$-point  rule is at most $2n-1$.

Let us now prove that $I(q_{n+k})=0$, for all  $0\leq k\leq n-1$.   For, consider the following basis  of $\mathbb{P}_{2n-1}$:
$$\tilde{B}=\{O_0, O_1, \ldots, O_n, O_n\varphi_1, O_n\varphi_2, \ldots, O_n \varphi_{n-1}\},$$
where $\{O_0, O_1, \ldots, O_n\}$ is an orthogonal polynomial basis of $\mathbb{P}_n$.
Any polynomial $q_{n+k}$ of degree $n+k$ (with $0\leq k\leq n-1$) can be written as a (unique) linear combination of the elements of $\tilde{B}$, that is
\begin{equation}\label{eq qnk}
q_{n+k}(x) =\sum_{i=0}^n\gamma_i O_i(x)+ \sum_{j=1}^k\gamma_{n+j} O_n(x)\varphi_j(x).
\end{equation}
For any integer $0\leq k\leq n-1$, the integral $I(q_{n+k} )$ is
\begin{align}\nonumber
I(q_{n+k} )&=\int_{a}^b w(x) q_{n+k} (x) dx=\langle q_{n+k} , \varphi_0\rangle\\\nonumber
&= \sum_{i=0}^n\gamma_i \langle O_i, \varphi_0\rangle+ \sum_{j=1}^k\gamma_{n+j} \langle O_n\varphi_j, \varphi_0\rangle\\ \label{eq qnkint}
&=\gamma_0 \langle O_0, \varphi_0\rangle+  \sum_{j=1}^k\gamma_{n+j} \langle O_n, \varphi_j\rangle =\gamma_0 \langle O_0, \varphi_0\rangle,
\end{align}
where  the last two equalities follow from  the fact that an orthogonal polynomial $O_j$ of degree $j$  is orthogonal to any polynomial of degree less than $j$. 

Evaluating \eqref{eq qnk} at the nodes $x_1, \ldots, x_n$, and taking into account  these nodes are roots of both $O_n$ and  $q_{n+k}$, we obtain a homogeneous linear system $Bx=0$, in the unknowns $\gamma_0, \gamma_1, \ldots, \gamma_{n-1}$, such that the first column of $B$ has all entries equal to 1. Thus,  $\gamma_0=0$ and from \eqref{eq qnkint} we get $I(q_{n+k} )=0$.

It remains to show that the error expression has the form given  in \eqref{error gaussL}. As  $x_1, \ldots, x_n$ are the roots of $O_n$, we have $O_n=\alpha (x-x_1)\cdots (x-x_n)$, where $\alpha$ is the coefficient of  $x^n$ in $O_n$. Then,
$$\|O_n(x)\|^2=\langle O_n(x), O_n(x)\rangle = \alpha^2\int_{a}^b w(x) O_n^2(x) dx=  \alpha^2 \int_{a}^b w(x) q_{2n}(x) dx.
$$
So, $I(q_{2n})=\frac{1}{\alpha^2}\|O_n(x)\|^2$ and  \eqref{error gaussL}  follows from \eqref{error prop}.
\end{proof}

\subsubsection*{Example: Gauss-Legendre rule}

The Legendre polynomials, usually denoted by $P_n$,  are orthogonal with respect to  the   weight function  $w(x)=1$ in   the integration interval   $[-1,1]$.
 The coefficient $\alpha$ and  $\|P_n\|^2$ in the error expression  \eqref{error gaussL}  satisfy the following   equalities (see, for instance Davis~\cite{davis}, p. 33):
\begin{equation}\label{last}
\alpha=\frac{(2n)!}{2^n (n!)^2}, \quad \|P_n\|^2=\frac{2}{2n+1}.
\end{equation}
Thus, from \eqref{error gaussL},  the  error expression for the Gauss-Legendre rule is
\begin{equation}\label{lasta}
E_{G_n}(f) = c_n f^{(2n)}(\xi)\quad\text{with $c_n=\frac{2^{2n+1} (n!)^4}{(2n+1)[(2n)!]^3}$}, \quad \xi\in(-1,1),
\end{equation}
agreeing with the well known expression for the error of a $n$-point Gauss-Legendre rule (see, for instance,  Atkinson~\cite{atkinson}, p. 276).

\section*{ Computational remarks}

The computational efficiency  of the algorithm referred in Remark~\ref{rem21}-2),   either  for the  rules previously discussed   or any other polynomial quadrature rule one may construct, deserve  further studies.   However,  we have yet   developed a {\sl Mathematica}  (Wolfram \cite{wolfram1})  code for the referred algorithm which, for a given number  $n$ of nodes,  produces  a list $\{nodes,weights, c_n\}$,  where $c_n$ is the coefficient in the respective error expression $E_n(f)$.  This code has been implemented for all  the rules detailed in the previous section and the computed coefficients $c_n$ agree with those in \eqref{errorNC}, \eqref{error AdamB} and \eqref{lasta}.   Whenever possible the elements of the mentioned  list  have been computed exactly, such as in the cases of Newton-Cotes, Adams-Bashford and Adams-Moulton rules. Our  program has   been also  tested for the Gauss-Legen\-dre rule for a wide range of nodes number. For this rule, when $2\leq n\leq 5$,  we   obtained  closed expressions for the respective   pairs $(W_n,c_n)$,  while for $6\leq n\leq 256$ we have  computed $100$\--digits of precision approximations for the   weights. The computations were carried out on a small laptop. The control of the error   in the computation of the weights  has been monitored  taking into account that the sum of the  weights  satisfy the equality $\sum_{i=1}^n a_i=2$, and knowing that all  the weights in $W_n$ are positive.  Our numerical results were compared with those  available in some published   tables:  for $2\leq n\leq 512$,   Stroud and Secrest  (\cite{stroud}, p. 99), $[30S]$;     for $2\leq n\leq 48 $, Krylov (\cite{krylov}, p. 337),  $[20S]$;  for  $2\leq n\leq 128$,  Evans  (\cite{evans}, p. 303),  $[30S]$;  for   $2\leq n\leq 16$,   Kythe and Sch\"{a}ferkotter  (\cite{kythe}, p. 505), $[16S]$ \---  the symbol $kS$ means $k$ significant digits  in the corresponding table. 

We consider our high precision  {\it in situ}  computation of nodes,  weights, and error coefficients  for the $n$-point  Gauss\--Legendre's rule   a tribute to the pioneering work of  Lowan, Davids and Levenson  \cite{lowan1} and  \cite{lowan2}, who managed  to publish a 16-digits precision Gauss-Legendre table, for $2\leq n\leq 16$, in the terrible years of WWII.

 \bigbreak
\noindent
{\bf Acknowledgments}

\noindent
MMG is supported by  Instituto de Mecânica\--IDMEC/IST, Centro de Projecto Mecânico, through FCT (Portugal)/program POCTI.  MESD is supported by the FCT through the Program POCI 2010/FEDER.


\end{document}